\documentclass[13pt, reqno]{amsart}
\usepackage{amsmath}
\usepackage{amscd}
\usepackage{amsthm}
\usepackage{amssymb} \usepackage{latexsym}
\usepackage{eufrak}
\usepackage{euscript}
\usepackage{epsfig}
\usepackage{graphics}
\usepackage{array}
\usepackage{enumerate}
\usepackage{color}
\usepackage{wasysym}
\usepackage{pdfsync}
\usepackage{stmaryrd}

\newcommand{\bel}[1]{\begin{equation}\label{#1}}

\newcommand{\be}{\begin{equation}}

\newcommand{\ba}{\begin{eqnarray}}
\newcommand{\ea}{\end{eqnarray}}

\newcommand{\qe}{\end{equation}}
\newcommand{\R}{{\mathbb R}}
\newcommand{\N}{{\mathbb N}}
\newcommand{\Z}{{\mathbb Z}}

\newcommand{\eg}{\begin{example}}
\newcommand{\egd}{\end{example}}
\newcommand{\tm}{\begin{thm}}
\newcommand{\tmd}{\end{thm}}
\newcommand{\co}{\begin{coro}}
\newcommand{\cod}{\end{coro}}
\newcommand{\enu}{\begin{enumerate}}
\newcommand{\enud}{\end{enumerate}}
\newcommand{\rmk}{\begin{rem}}
\newcommand{\rmkd}{\end{rem}}

\theoremstyle{theorem}
\newtheorem{thm}{Theorem}[section]
\newtheorem{prop}{Proposition}[section]
\theoremstyle{example}
\newtheorem{example}{Example}[section]
\theoremstyle{corollary}
\newtheorem{coro}{Corollary}[section]
\theoremstyle{lemma}
\newtheorem{lemma}{Lemma}[section]
\theoremstyle{definition}

\theoremstyle{proof}

\theoremstyle{remark}
\newtheorem{rem}{Remark}[section]
\theoremstyle{remark}

\begin{document}

\title[Ancient caloric functions on graphs]{Dimensional bounds for ancient caloric functions on graphs}

\author{Bobo Hua}
\address{School of Mathematical Sciences, LMNS,
Fudan University, Shanghai 200433, China; Shanghai Center for
Mathematical Sciences, Fudan University, Shanghai 200433,
China.}
\email{bobohua@fudan.edu.cn}

%\author{Lili Wang}
%\address{School of Mathematical Sciences, Fudan University, Shanghai 200433, China.}
%\email{lilyecnu@outlook.com}
%\thanks{FB was partially supported by the Alexander von Humboldt foundation and partially supported by the NSF Grant DMS-0804454 Differential Equations in Geometry.}
\begin{abstract}
We study ancient solutions of polynomial growth to heat equations on graphs, and extend Colding and Minicozzi's theorem \cite{ColdingM19} on manifolds to graphs: For a graph of polynomial volume growth, the dimension of the space of ancient solutions of polynomial growth is bounded by the product of the growth degree and the dimension of harmonic functions with the same growth.\end{abstract}

\maketitle
%\section{Response to the report}

%\begin{proof}[Proof of Example~1.4]
%We may refer to my joint paper with Mugnolo, ``Time regularity and long-time behavior of parabolic p-Laplace equations on infinite graphs", see Theorem~3.6.
%\end{proof}

%\begin{proof}[Proof of Example~1.5]
%By the assumption, we have $(S_p)$ for the measure $n,$ that is,
%\begin{equation}\label{eq:ef}\frac{S}{2}\sum_{x,y\in X}|f(x)-f(y)|^2b(x,y)\geq \left(\sum_{x\in X}|f(x)|^p n(x)\right)^{\frac2p}.\end{equation} Note that $\inf_{x\in X} n(x)\geq c_0>0$ and $\sup_{x\in X}m(x)\leq c_1<\infty,$ where the second assertion follows from $\sum_{x\in X}m(x)<\infty.$ Hence $$n(x)\geq \frac{c_0}{c_1}m(x),\quad \forall x\in X.$$ Hence \eqref{eq:ef} yields 
%\begin{equation*}\frac{S}{2}\sum_{x,y\in X}|f(x)-f(y)|^2b(x,y)\geq (\frac{c_0}{c_1})^{\frac2p}\left(\sum_{x\in X}|f(x)|^pm(x)\right)^{\frac2p}.\end{equation*}

%\end{proof}

\section{Introducton}

For a (complete) Riemannian manifold $M$ and $k>0,$ we denote by $\mathcal{H}_k(M)$ the space of harmonic functions $u$ satisfying that there are some $p\in M$ and a constant $C_f$ depending on $f$ such that
$$\sup_{x\in B_R(p)}|f(x)|\leq C_f(1+R)^k,\quad \forall\  R\geq 1,$$ where $B_R(p)$ is the ball of radius $R$ centered at $p.$ This is the space of harmonic functions of polynomial growth on $M$ with the growth degree at most $k.$

Let $N$ be a Riemannian manifold with nonnegative Ricci curvature. In 1975, Yau \cite{YauCPAM75} proved the Liouville theorem that any positive harmonic function on $N$ is constant. Yau conjectured that for any $k>0$ the space $\mathcal{H}_k(N)$ is finite-dimensional, see e.g. \cite{Yau87EM,Yau93book}. This conjecture was settled in \cite{ColdingMAnnals97}, see also \cite{ColdingMJDG97,ColdingMInv98,ColdingMCPAM98, LiMRL98, CCM95, LiTam89} for related results. 

A natural generalization is to consider ancient solutions, defined on the time interval $(-\infty,0],$ of polynomial growth to heat equations. For a Riemannian manifold $M$ and $k>0,$ we denote by $\mathcal{P}_k(M)$ the space of ancient solutions $u(x,t)$ satisfying that
there are some $p\in M$ and a constant $C_u$ such that
$$\sup_{B_R(p)\times [-R^2, 0]}|u|\leq C_u(1+R)^k,\quad \forall\  R\geq 1.$$
Calle \cite{CalleMZ06,Callethesis} initiated the study of dimensional bounds for $\mathcal{P}_k(M).$ For an $n$-dimensional Riemannian manifold $N$ with nonnegative Ricci curvature, Lin and Zhang \cite{LinZhang17} proved that
$$\dim\mathcal{P}_k(N)\leq C k^{n+1}, \quad k\geq 1.$$ Recently, Colding and Minicozzi \cite{ColdingM19} proved the following general result, which yields the improvement of Lin and Zhang's result, $$\dim\mathcal{P}_k(N)\leq C k^{n}, \quad k\geq 1.$$
\tm[\cite{ColdingM19}]\label{thm:CM1} If there are $p\in M,$ constants $C, d_V$ such that $$\mathrm{Vol}(B_R(p))\leq C(1+R)^{d_V},\quad \forall R\geq 1,$$ then
$$\dim \mathcal{P}_{2k}(M)\leq (k+1)\dim \mathcal{H}_{2k}(M),\quad \forall k\geq 1.$$
\tmd
%\footnote{In fact, they proved the dimension estimate for the manifolds satisfying the volume doubling property and  the mean value inequality for heat equations.}

Harmonic functions of polynomial growth on graphs have been extensively studied by many authors, e.g. \cite{DelmottePolynomial98,KleinerJAMS10,ShalomTao10,Taoweb2,HJLAGAG13,HornLinLiuYau14,BDKY15,HuaJostLiu15,HuaJostMathZ15,HuaJostTAMS15,MPTY17}. In this paper, we study ancient solutions of polynomial growth to heat equations on graphs. We will generalize Colding and Minicozzi's theorem, Theorem~\ref{thm:CM1}, to discrete spaces, as proposed in \cite{ColdingM19}. \footnote{``We expect that the proof of Theorem 0.3 (see Theorem~\ref{thm:CM1} in this paper) extends to many discrete spaces, allowing a wide range of applications", quoted from \cite{ColdingM19}.}
%Moreover, these properties make sense also for discrete spaces, vastly extending the theory and methods out of the continuous world. Recently Kleiner, [K], (see also Shalom-Tao, [ST], [T1], [T2]) used, in part, this in his new proof of an important and foundational result in geometric group theory, originally due to Gromov, [G]. We expect that the proof of Theorem 0.3 extends to
%many discrete spaces, allowing a wide range of applications.

We recall the setting of weighted graphs. Let $(V,E)$ be a locally finite, simple, undirected graph. Two vertices $x,y$ are called neighbours, denoted by $x\sim y$, if there is an edge connecting $x$ and $y,$ i.e. $\{x,y\}\in E.$ 
We denote by $d$ the usual combinatorial graph distance, that is, $d(x, y) :=
\inf\{n | x = z_0 \sim. . . \sim z_n = y\},$ and by $$B_R(x):=\{y\in V: d(y,x)\leq R\}$$ the ball of radius $R>0$ centered at $x.$
Let $$w: E\to (0,\infty),\quad \{x,y\}\mapsto w_{xy}=w_{yx},$$ be the edge weight function. For any vertex $x,$ the weighted vertex degree is defined as
$$\mu_x:=\sum_{y\in V:y\sim x}w_{xy}.$$ Hence $(V,\mu)$ can be regarded as a discrete measure space. For any $\Omega\subset V,$ we denote by $\mu(\Omega):=\sum_{x\in \Omega}\mu_x$ the $\mu$-measure of $\Omega.$ We call the triple $G=(V,E,w)$ a weighted graph.

For a weighted graph $G$ and any function $f:V\to \R,$ the Laplace operator $\Delta$ is defined as
$$\Delta f (x):= \sum_{y\in V:y\sim x}\frac{w_{xy}}{\mu_x}\left(f(y)-f(x)\right), \quad\forall x\in V.$$
A function $f$ on $V$ is called harmonic if $\Delta f=0.$ Let $\R_-:=(-\infty,0].$ A function $u(x,t)$ on $V\times \R_-$ is called an ancient solution to the heat equation if $$\frac{\partial }{\partial t} u(x,t)=\Delta u(x,t),\quad \forall x\in V, t\in \R_-.$$

 We denote by $\mathcal{H}_k(G)$ the space of harmonic functions of polynomial growth on $G$ with the growth degree at most $k,$ i.e. $f\in \mathcal{H}_k(G)$ if $\Delta f=0$ and there are some $x_0\in V$ and a constant $C_f$ such that
 $$\sup_{x\in B_R(x_0)}|f(x)|\leq C_f(1+R)^k.$$ 
We denote by $\mathcal{P}_k(G)$  the space of ancient solutions of polynomial growth to heat equation with the growth degree at most $k,$ i.e. $u\in \mathcal{P}_k(G)$ if $\partial_t u=\Delta u$ on $V\times \R_-$ and for some $x_0\in V$ and a constant $C_u,$ such that
 $$\sup_{(x,t)\in B_R(x_0)\times[-R^2, 0]}|u(x,t)|\leq C_u(1+R)^k.$$ 

The following is the main result of the paper.
\tm\label{thm:main1} For a weighted graph $G,$ if for some $x_0\in V,$ $\alpha>0$ and $C>0,$ such that $$\mu(B_R(x_0))\leq C(1+R)^{\alpha},\quad \forall R\geq 1,$$ then for all $k\geq 1,$
$$\dim \mathcal{P}_{2k}(G)\leq (k+1)\dim \mathcal{H}_{2k}(G).$$
\tmd
\rmk \enu
\item This extends Theorem~\ref{thm:CM1} to the discrete setting.
\item By this theorem, one can derive dimensional bounds of ancient solutions of polynomial growth via those of harmonic functions of polynomial growth. By applying this result, we get dimensional bounds of ancient solutions of polynomial growth on many graphs, such as graphs satisfying the volume doubling property and the Poincar\'e inequality, e.g.  Cayley graphs of nilpotent groups, graphs satisfying the curvature dimension condition $CDE'(0,n)$ for some finite $n$ \cite{HornLinLiuYau14}, and planar graphs with nonnegative combinatorial curvature \cite{HuaJostLiu15}.
\item By the calculation of caloric polynomials on the integer lattices $\Z^n,$  the estimate in Theorem~\ref{thm:main1} is sharp in the order of $k,$ $k\to \infty,$ see the Appendix.
\enud
\rmkd

For the proof of the theorem, we closely follow the arguments in Lin and Zhang \cite{LinZhang17}, and Colding and Minicozzi \cite{ColdingM19}: We first prove the Caccioppoli type inequality to the heat equation on graphs, and use it to conclude that higher order time derivatives of an ancient solution $u$ of polynomial growth vanish. This yields a decomposition of $u,$ i.e. $u$ is a polynomial in time. Then a tricky dimensional counting argument implies the result. Our contribution is the proof of a discrete analog of the Caccioppoli type inequality to the heat equation, see Theorem~\ref{thm:paracacc}. For the desired estimate, some additional terms appear in the discrete setting, such as $$\sum_{x,y\in V: y\sim x}w_{xy}|u(y,0)-u(x,0)|^2,$$ see \eqref{eq:et1}. Its continuous counterpart is $\int |\nabla u|^2$ for $t=0.$ Using the discrete feature, which is related to the boundedness of the Laplacian on graphs, we estimate it by the quantity $\sum_{x\in V}\mu_x|u(x,0)|^2.$ This quantity was unnecessary, and hence dropped in the continuous setting, see \eqref{eq:est1} and \cite[(1.4) and (1.5)]{ColdingM19}.

The paper is organized as follows: In the next section, we prove the Caccioppoli type inequality for the heat equation on graphs. In Section~\ref{sec:proof}, we give the proof of Theorem~\ref{thm:main1}. In the Appendix, we calculate the dimension of caloric polynomials on $\Z^n.$ 

In this paper, for simplicity the constant $C$ may change from line to line.

\section{Caccioppoli type inequality for heat equations}
Let $G=(V,E,w)$ be a weighted graph. For convenience, we extend the edge weight function $w$ to $V\times V$ by
setting $w_{xy}=0$ for any $x\not\sim y.$ In this way, for a function $f$ on $V$ we may write
$$\sum_{y\in V}w_{xy}f(y)=\sum_{y\in V: y\sim x} w_{xy}f(y).$$ For any $\Omega\subset V,$ we write 
$$\sum_\Omega f:=\sum_{x\in \Omega}f(x)\mu_x,\ \sum f:=\sum_{x\in V}f(x)\mu_x,$$ whenever they make sense.
The difference operator $\nabla$ is defined as
$$\nabla_{xy}f=f(y)-f(x),\quad \forall x,y\in V.$$ The following proposition is elementary.
\begin{prop} \begin{equation}\label{eq:basic}\nabla_{xy}(fg)=f(x)\nabla_{xy}g+g(y)\nabla_{xy}f.
\end{equation}
\end{prop}

The ``carr\'e du champ" operator $\Gamma$ is defined as 
$$\Gamma(f)(x)=\frac12\sum_{y\in V}\frac{w_{xy}}{\mu_x}(f(y)-f(x))^2,\quad x\in V.$$ So that $\Gamma(f)$ is a function on $V.$ This is a discrete analog of the norm squared of the gradient of a $C^1$ function $f$ on a manifold, i.e. $|\nabla f|^2.$

%We denote by $C_0(V)$ the set of functions of finite support. 

The following Green's formula is well-known, see e.g. \cite[Theorem~2.1]{Grigoryanbook}.
\tm For any $f,g:V\to \R,$ if $g$ is of finite support, then
\begin{equation}\label{eq:Green}\frac12\sum_{x,y}w_{xy}\nabla_{xy}f\nabla_{xy}g=-\sum_{x\in V}\Delta f(x) g(x)\mu_x.\end{equation}
\tmd

From now on, we fix $x_0\in V$ as a base vertex. We write $B_R:=B_R(x_0),$ $R>0,$ for simplicity. We denote by $Q_R:=B_R\times [-R^2,0]$ the parabolic cylinder of size $R$ at $(x_0, 0).$
For a space-time function $u(x,t)$ on $V\times \R_-,$ 
we denote 
$$\int_{Q_R} u:=\int_{-R^2}^{0}\sum_{x\in B_R}u(x,t)\mu_xdt.$$ For any $t_0\in \R_-,$
we write
$$\left.\sum_\Omega u\right|_{t=t_0}:=\sum_{x\in \Omega}u(x,t_0)\mu_x.$$

The following is the Caccioppoli type inequality to the heat equation on graphs, see e.g. \cite[(3.12)]{LinZhang17} and \cite[(1.2)]{ColdingM19} for the continuous setting. 
\tm\label{thm:paracacc} There is a universal constant $C$ such that for any ancient solution $u_t=\Delta u$ and $R\geq 1,$
\begin{equation}\label{eq:cacci1} R^2\int_{Q_R}\Gamma(u)+R^4\int_{Q_R}u_t^2\leq C\int_{Q_{36R}} u^2.
\end{equation}
\tmd

\begin{proof} We follow the proof strategy in the continuous setting, see e.g. \cite{ColdingM19}. Some modifications for the discrete setting are needed. For any $R>0,$ we denote by 
$$\eta(x):=0\vee \left(2-\frac{d(x,x_0)}{R}\right)\wedge 1$$ the cut-off function on $B_{2R}.$ It is easy to see that $\eta$ is supported in $B_{2R},$ and takes constant-value $1$ on $B_R.$ Moreover, $|\nabla_{xy}\eta|\leq \frac{2}{R},$ for any $x\sim y.$ 

We first estimate $\int_{Q_R}\Gamma(u).$ Multiplying $4\eta^2u$ on both sides of the heat equation $u_t=\Delta u,$ and summing over the space $V,$ we have
\begin{eqnarray*}
&&\frac{d}{dt}\left(2\sum\eta^2 u^2\right)\\&=&\sum 4\eta^2 u\partial_tu=2\sum\eta^2u\Delta u
=-\sum_{x,y}w_{xy}\nabla_{xy}u\nabla_{xy}(\eta^2 u)\\
&=&-\sum_{x,y}w_{xy}\nabla_{xy}u(\eta^2(x)\nabla_{xy}u+u(y,t)\nabla_{xy}(\eta^2))\\
&=&-\sum_{x,y}w_{xy}|\nabla_{xy}u|^2\eta^2(x)-\sum_{x,y}w_{xy}u(y,t)\nabla_{xy}u\nabla_{xy}\eta(2\eta(x)+\nabla_{xy}\eta)\\
&=&-\sum_{x,y}w_{xy}|\nabla_{xy}u|^2\eta^2(x)-2\sum_{x,y}w_{xy}\eta(x)u(y,t)\nabla_{xy}u\nabla_{xy}\eta-\sum_{x,y}w_{xy}u(y,t)\nabla_{xy}u|\nabla_{xy}\eta|^2.
\end{eqnarray*} where we used the Green's formula \eqref{eq:Green} in the second line.
For the last term on the right hand side of the above inequality, by swapping $x$ and $y,$ the symmetry yields that
\begin{eqnarray*}-\sum_{x,y}w_{xy}u(y,t)\nabla_{xy}u|\nabla_{xy}\eta|^2&=&-\frac{1}{2}\sum_{x,y}w_{xy}(u(y,t)-u(x,t))\nabla_{xy}u|\nabla_{xy}\eta|^2\\
&=&-\frac12\sum_{x,y}w_{xy}|\nabla_{xy}u|^2|\nabla_{xy}\eta|^2\leq 0.
\end{eqnarray*} 
Dropping this term, we get
\begin{eqnarray*}
\frac{d}{dt}\left(2\sum\eta^2 u^2\right)&\leq&-\sum_{x,y}w_{xy}|\nabla_{xy}u|^2\eta^2(x)-2\sum_{x,y}w_{xy}\eta(x)u(y,t)\nabla_{xy}u\nabla_{xy}\eta\\
&\leq&-\sum_{x,y}w_{xy}|\nabla_{xy}u|^2\eta^2(x)+\frac12\sum_{x,y}w_{xy}|\nabla_{xy}u|^2\eta^2(x)+2\sum_{x,y}w_{xy}u^2(y,t)|\nabla_{xy}\eta|^2\\
&=&-\frac12\sum_{x,y}w_{xy}|\nabla_{xy}u|^2\eta^2(x)+2\sum_{x,y}w_{xy}u^2(y,t)|\nabla_{xy}\eta|^2.
\end{eqnarray*} where we used $2ab\leq \frac12 a^2+2b^2.$
For $R\geq 1,$ integrating this in time from $-R^2$ to $0,$ and using the properties of $\eta,$ we get
\begin{eqnarray}
2\left.\sum_{B_R} u^2\right|_{t=0}+\int_{Q_R}\Gamma(u)&\leq&2\left.\sum\eta^2 u^2\right|_{t=0}+\frac12\int_{-R^2}^0\sum_{x,y}w_{xy}|\nabla_{xy} u|^2\eta^2(x)\nonumber\\
&\leq &2\int_{-R^2}^0\sum_{x,y}w_{xy}u^2(y,t)|\nabla_{xy}\eta|^2dt+2\left.\sum \eta^2u^2\right|_{t=-R^2}\nonumber\\
&\leq&\frac{8}{R^2}\int_{-R^2}^0\sum_{y\in B_{2R+1}}\mu_yu^2(y,t)dt+2\left.\sum_{B_{2R}} u^2\right|_{t=-R^2}\nonumber\\
&\leq &\frac{8}{R^2}\int_{Q_{3R}}u^2+2\left.\sum_{B_{2R}} u^2\right|_{t=-R^2}.\label{eq:est1}
\end{eqnarray}
By the mean value property, there is $R_1\in [R,2R]$ such that
$$\left.\sum_{B_{4R}}u^2\right|_{t=-R_1^2}=\frac{1}{3R^2}\int_{-4R^2}^{-R^2}\sum_{x\in B_{4R}}\mu_xu^2(x,t)dt.$$ By using \eqref{eq:est1} for $R=R_1$ and the above equation, we get
\begin{eqnarray}
2\left.\sum_{B_R} u^2\right|_{t=0}+\int_{Q_R}\Gamma(u)&\leq& 2\left.\sum_{B_{R_1}} u^2\right|_{t=0}+\int_{Q_{R_1}}\Gamma(u)
\leq \frac{8}{R_1^2}\int_{Q_{3R_1}}u^2+2\left.\sum_{B_{2R_1}} u^2\right|_{t=-R_1^2}\nonumber\\
&\leq &\frac{C}{R^2}\int_{Q_{6R}}u^2+\frac{2}{3R^2}\int_{-4R^2}^{-R^2}\sum_{x\in B_{4R}}\mu_xu^2(x,t)dt\nonumber\\
&\leq&\frac{C}{R^2}\int_{Q_{6R}}u^2.\label{eq:est2}
\end{eqnarray} Note that the term $\left.\sum_{B_R} u^2\right|_{t=0}$ on the left hand side of the inequality is not needed in the continuous setting, see \cite[(1.5)]{ColdingM19}. We keep this term for the following estimates.

Next we estimate $\int_{Q_R}u_t^2.$ By differentiating in time, and by Green's formula \eqref{eq:Green}, we get
\begin{eqnarray*}
\frac{d}{dt}\left(\sum\Gamma(u)\eta^2\right)&=&\sum_{x,y}w_{xy}\nabla_{xy}u\nabla_{xy}u_t\eta^2(x)\\
&=&\sum_{x,y}w_{xy}\nabla_{xy}u\left(\nabla_{xy}(u_t\eta^2)-u_t(y,t)\nabla_{xy}(\eta^2)\right)\\
&=&-2\sum(\Delta u)u_t\eta^2-\sum_{x,y}w_{xy}u_t(y,t)\nabla_{xy}u\nabla_{xy}\eta(2\eta(y)-\nabla_{xy}\eta)\\
&=&-2\sum u_t^2\eta^2-2\sum_{x,y}w_{xy}u_t(y,t)\eta(y)\nabla_{xy}u\nabla_{xy}\eta+\sum_{x,y}w_{xy}u_t(y,t)\nabla_{xy}u|\nabla_{xy}\eta|^2\\
&=:&-2\sum u_t^2\eta^2+I+II.
\end{eqnarray*}
For the term $I$ on the right hand side of the above inequality, by $2ab\leq a^2+b^2,$
\begin{eqnarray*}
I&\leq& \sum_{x,y}w_{xy}u_t^2(y,t)\eta^2(y)+\sum_{x,y}w_{xy}|\nabla_{xy}u|^2|\nabla_{xy}\eta|^2\\
&=&\sum u_t^2\eta^2+\sum_{x,y}w_{xy}|\nabla_{xy}u|^2|\nabla_{xy}\eta|^2.
\end{eqnarray*}
For the term $II,$ by swapping $x$ and $y,$ the symmetry yields that
\begin{eqnarray*}
II&=&\frac12\sum_{x,y}w_{xy}\nabla_{xy}u\nabla_{xy} u_t|\nabla_{xy}\eta|^2\\
&=&\frac14\frac{d}{dt}\left(\sum_{x,y}w_{xy}|\nabla_{xy}u|^2|\nabla_{xy}\eta|^2\right)
\end{eqnarray*}
Combining the above estimates, we get
$$\frac{d}{dt}\left(\sum\Gamma(u)\eta^2\right)\leq -\sum u_t^2\eta^2+\sum_{x,y}w_{xy}|\nabla_{xy}u|^2|\nabla_{xy}\eta|^2+\frac14\frac{d}{dt}\left(\sum_{x,y}w_{xy}|\nabla_{xy}u|^2|\nabla_{xy}\eta|^2\right).$$
For $R\geq 1,$ integrating this in time from $-R^2$ to $0,$ and using the properties of $\eta,$ we have
\begin{eqnarray}
\int_{Q_R}u_t^2&\leq&\int_{-R^2}^0\sum_{x,y}w_{xy}|\nabla_{xy}u|^2|\nabla_{xy}\eta|^2+\left.\frac14\sum_{x,y}w_{xy}|\nabla_{xy}u|^2|\nabla_{xy}\eta|^2\right|_{t=0}+\left.\sum\Gamma(u)\eta^2\right|_{t=-R^2}\nonumber\\
&=:&III+IV+\left.\sum\Gamma(u)\eta^2\right|_{t=-R^2}.\label{eq:et1}
\end{eqnarray} For the term $III$ in the above inequality,
\begin{eqnarray*}
III&=&\int_{-R^2}^0\sum_{x,y}w_{xy}|\nabla_{xy}u|^2|\nabla_{xy}\eta|^2\leq \frac{4}{R^2}\int_{-R^2}^0\sum_{x,y\in B_{2R+1}}w_{xy}|\nabla_{xy}u|^2
\\
&\leq&\frac{C}{R^2}\int_{Q_{3R}}\Gamma(u)\leq \frac{C}{R^4}\int_{Q_{18R}}u^2,
\end{eqnarray*} where we used \eqref{eq:est2} in the last inequality.

For the term $IV,$ it is an additional term appearing in the discrete setting. We estimate it as follows,
\begin{eqnarray*}IV&=&\left.\frac14\sum_{x,y}w_{xy}|\nabla_{xy}u|^2|\nabla_{xy}\eta|^2\right|_{t=0}\leq\left. \frac{1}{R^2}\sum_{x,y\in B_{2R+1}}w_{xy}|\nabla_{xy}u|^2\right|_{t=0}\\
&\leq&\left. \frac{2}{R^2}\sum_{x,y\in B_{3R}}w_{xy}(u(x,t)^2+u(y,t)^2)\right|_{t=0}\leq\left. \frac{4}{R^2}\sum_{ B_{3R}}u^2\right|_{t=0}\\
&\leq& \frac{C}{R^4}\int_{Q_{18R}}u^2,
\end{eqnarray*} where the term $\left.\sum_{B_{3R}} u^2\right|_{t=0}$ is bounded by using \eqref{eq:est2}. Hence by \eqref{eq:et1},
\begin{equation}\label{eq:et2}\int_{Q_R}u_t^2\leq \frac{C}{R^4}\int_{Q_{18R}}u^2+\left.\sum_{B_{2R}}\Gamma(u)\right|_{t=-R^2}.
\end{equation}
By the mean value property, there is $R_2\in [R,2R]$ such that
\begin{equation}\label{eq:et3}\left.\sum_{B_{4R}}\Gamma(u)\right|_{t=-R_2^2}=\frac{1}{3R^2}\int_{-4R^2}^{-R^2}\sum_{x\in B_{4R}}\mu_x\Gamma(u)(x,t)dt.\end{equation}
Applying \eqref{eq:et2} for $R=R_2,$ we get
\begin{eqnarray}
\int_{Q_R}u_t^2&\leq&\int_{Q_{R_2}}u_t^2\leq \frac{C}{R_2^4}\int_{Q_{18R_2}}u^2+\left.\sum_{B_{2R_2}}\Gamma(u)\right|_{t=-R_2^2}\nonumber\\
&\leq &\frac{C}{R^4}\int_{Q_{36R}}u^2+\left.\sum_{B_{4R}}\Gamma(u)\right|_{t=-R_2^2}\nonumber\\
&= &\frac{C}{R^4}\int_{Q_{36R}}u^2+\frac{1}{3R^2}\int_{-4R^2}^{-R^2}\sum_{x\in B_{4R}}\mu_x\Gamma(u)(x,t)dt\nonumber\\
&\leq &\frac{C}{R^4}\int_{Q_{36R}}u^2+\frac{C}{R^2}\int_{Q_{4R}}\Gamma(u)\nonumber\\
&\leq &\frac{C}{R^4}\int_{Q_{36R}}u^2,\label{eq:et4}
\end{eqnarray} where we have used \eqref{eq:et2} and \eqref{eq:est2}.

The theorem follows from \eqref{eq:est2} and \eqref{eq:et4}.
\end{proof}

This yields the following corollary.
\co\label{coro:c1} Let $u\in \mathcal{P}_{k}(G).$ If there are constants $\alpha,C>0$ such that 
$$\mu(B_R)\leq C(1+R)^\alpha,\quad \forall R\geq 1,$$ then for any $m\in \N,$ $4m>2k+\alpha+2,$
$$\partial_t^m u\equiv 0.$$
\cod
\begin{proof}
We follow the argument in \cite{ColdingM19}\footnote{This result was proved for the manifolds satisfying stronger conditions, the volume doubling property and the mean value inequality for heat equations, see \cite[Page 17]{LinZhang17}. }. For the sake of completeness, we include the proof here. 

Since $\partial_t-\Delta$ commutes with $\partial_t,$ $$(\partial_t-\Delta)\partial_t^i u=0,\quad\forall i\in\N.$$ For any $R\geq 1,$ applying Theorem~\ref{thm:paracacc} for $\partial_t^i u,$ $0\leq i\leq m-1,$
we get
\begin{eqnarray*}
\int_{Q_R}|\partial_t^mu|^2&\leq& \frac{C}{R^4}\int_{Q_{36R}} |\partial_t^{m-1}u|^2\\
&\leq &\cdots\leq \frac{C(m)}{R^{4m}}\int_{Q_{(36)^mR}} u^2\\
&\leq& \frac{C}{R^{4m}}\mu(B_{(36)^mR}) (36)^{2m}R^2 \sup_{Q_{(36)^mR}}u^2\\
&\leq& C R^{-4m+2k+\alpha+2}. 
\end{eqnarray*} So that for $4m>2k+\alpha+2,$ by passing to the limit $R\to +\infty,$ we get
$$\partial_t^m u\equiv 0.$$ This proves the corollary.

\end{proof}

\section{Proof of Theorem~\ref{thm:main1}}\label{sec:proof}
In this section, we prove the main theorem, Theorem~\ref{thm:main1}. The proof follows verbatim from \cite{ColdingM19}. We include it here for the sake of completeness.
Choose $m\in \N$ such that $4m>4k+\alpha+2.$ For any $u\in \mathcal{P}_{2k}(G),$
by Corollary~\ref{coro:c1}, $\partial_t^m u\equiv 0.$ Hence we have
$$u(x,t)=p_0(x)+tp_1(x)+\cdots+t^{m-1}p_{m-1}(x).$$ 
This yields that $u$ is a polynomial in time, which is crucial for the following arguments. This was obtained by \cite[Theorem~1.2]{LinZhang17} for manifolds satisfying the volume doubling property and  the mean value inequality for heat equations.  It was observed by \cite{ColdingM19} that the polynomial volume growth condition is in fact sufficient.

Furthermore, by the growth condition of $u\in \mathcal{P}_{2k}(G),$ considering large $t$ and fixed $x,$ we have $$p_i\equiv 0, \quad \forall i>k.$$ This yields that
\begin{equation}\label{eq:eqpq1}u(x,t)=p_0(x)+tp_1(x)+\cdots+t^{l}p_{l}(x),
\end{equation} where $l:=\lfloor k\rfloor,$ the greatest integer less than or equal to $k.$ 

We claim that the function $p_i(x),$ $0\leq i\leq l,$ grows polynomially with the growth degree less than or equal to $2k.$ Fix distinct values $-1<t_1<t_2<\cdots<t_l<t_{l+1}=0,$ by the computation of the Vandermonde determinant, one can show that
$$\beta_j:=(1,t_j,t_j^2,\cdots, t_j^l), \quad 1\leq j\leq l+1,$$ are linear independent in $\R^{l+1}.$ Let $e_i$ be the standard unit vector in $\R^{n+1}.$ Then there are $b_{j}^i\in \R$ such that
$$e_i=\sum_{j=1}^{l+1} b_j^i\beta_j.$$ Using this fact and \eqref{eq:eqpq1}, we get
$$p_i(x)=\sum_{j=1}^{l+1}b_j^i u(x,t_j).$$ Since $u(x,t_j),$ $1\leq j\leq l+1,$ grows polynomially with the growth degree less than or equal to $2k,$ so does $p_i.$ This proves the claim.

Since $u_t=\Delta u,$ by \eqref{eq:eqpq1}, 
\begin{equation*} \Delta p_l=0,\ \Delta p_i=(i+1)p_{i+1}, \quad 0\leq i\leq l-1.
\end{equation*} Hence we get a linear map
\begin{eqnarray*}\Psi_0:&&\mathcal{P}_{2k}(G)\to \mathcal{H}_{2k}(G)\\
&&\ \  \quad\quad u\mapsto p_l.
\end{eqnarray*} Let $\mathcal{K}_0:=\mathrm{Ker}(\Psi_0).$ It follows that
$$\dim \mathcal{P}_{2k}(G)\leq \dim \mathcal{K}_0+\dim \mathcal{H}_{2k}(G).$$
To estimate $\dim \mathcal{K}_0,$ we note that
for any $u\in \mathcal{K}_0,$ $$p_l=0, \ \Delta p_{l-1}=0.$$ Hence we have a linear map
\begin{eqnarray*}\Psi_1:&&\mathcal{K}_0\to \mathcal{H}_{2k}(G)\\
&&\ \  u\mapsto p_{l-1}.
\end{eqnarray*} Let $\mathcal{K}_1:=\mathrm{Ker}(\Psi_1).$  This yields that
$$\dim \mathcal{K}_0\leq \dim \mathcal{K}_1+\dim \mathcal{H}_{2k}(G).$$ Repeating this $l+1$ times, we prove that
\begin{eqnarray*}\dim \mathcal{P}_{2k}(G)&\leq &(l+1) \mathcal{H}_{2k}(G)\\
&\leq &(k+1) \mathcal{H}_{2k}(G).
\end{eqnarray*} This proves the theorem.

%\end{proof}

\textbf{Acknowledgements.} We thank Qi S. Zhang for many discussions and comments on dimension estimates of ancient solutions of polynomial growth. The author is supported by NSFC, no.11831004 and no. 11826031.

\section{Appendix: Caloric polynomials on $\Z^n$}
%Let $u(x,t)$ be a solution to the heat equation on $\R^n,$
%$$(\partial_t -\Delta) u(x,t)=0,\quad \forall x\in \R^n, t\in\R_-:=(-\infty,0].$$
%The ancient solution $u$ to the heat equation is called of polynomial growth if there are constant $C,d$ such that
%\begin{equation}\label{eq:eq1}|u(x,t)|\leq C(|x|+\sqrt{t}+1)^d,\forall x\in \R^n, t\in \R_-.\end{equation}

Let $(\Z^n, S)$ be the Cayley graph of the free abelian group $\Z^n$ with a finite symmetric generating set $S,$ i.e. $S=S^{-1},$ where each edge has unit weight. We write $\mathcal{P}_k^n=\mathcal{P}_k(\Z^n, S).$ 
%For any $d>0,$ we denote by $F^d$ the space of smooth functions on $\R^n\times\R_-$ satisfying $\eqref{eq:eq1}$ for some $C>0.$
%We denote by $\mathcal{P}_k^n$ the space of ancient solutions in $F^d.$ This is the space of polynomial growth ancient solutions with growth rate at most $d.$
%For any $d>0,$ we denote by 
%$$:=\{u(x,t): (\partial_t -\Delta) u=0\ \mathrm{on}\ \R^n\times\R_-, \eqref{eq:eq1}\ \mathrm{holds\ for\ some\ }C\}.$$

Let $\Z_+:=\N\cup\{0\}.$ For any $k\in \Z_+,$ we denote by $P^k$ the space of polynomials of degree less than or equal to $k$ in $\R^n.$
For any polynomial $u$ on $\R^n\times \R,$ we can write it as 
$$u(x,t)=\sum_{i=0}^mp_i(x)t^i,$$ where $p_i$ are polynomials in $x\in \R^n.$ The parabolic degree of $u$ is defined as $$\deg(u):=\max_{0\leq i\leq m}\{\deg(p_i)+2i\}.$$ For any $k\in \Z_+,$ we denote by 
$\hat{P}^k$ the space of polynomials in $(x,t)\in \R^{n}\times \R$ of parabolic degree at most $k.$ By the induction argument and the Liouville theorem, following \cite{HJLAGAG13}, one can prove that for any $k>0,$ $$\mathcal{P}_k^n\subset \hat{P}^{\lfloor k\rfloor}.$$
This means that for any $u\in \mathcal{P}_k^n,$ which is defined on $\Z^n\times \R_-,$ one can extend it to be a polynomial in $\R^n\times \R.$ By this result, to estimate the dimension of $\mathcal{P}_k^n$ it suffices to consider $k\in \Z_+$ and caloric polynomials, i.e. polynomials in $(x,t)$ satisfying the discrete heat eqaution.

%The following theorem is well-known.
%\begin{thm}For any $d>0,$ $$\mathcal{P}_k^n\subset \hat{P}^d.$$
%\end{thm}

In the following, we calculate the dimension of caloric polynomials on $(\Z^n, S).$ 
\begin{thm}\label{thm:app1}For any $k\in \Z_+,$ $$\dim \mathcal{P}_k^n=\dim P^k=\sum_{i=0}^k\binom{i+n-1}{n-1}.$$% For any $d>0,$ 
%$$\dim H^d_n=\sum_{i=0}^d\binom{i+n-1}{n-1}.$$
\end{thm} 

This yields the following corollary. 
\begin{coro}There are positive constants $C_1,C_2$ such that $$C_1 k^n\leq \dim \mathcal{P}_k^n\leq C_2 k^n,\quad \forall k\geq 1.$$
\end{coro} Noting that by Theorem~\ref{thm:main1} and \cite[Theorem~4]{HJLAGAG13}, we have
$$\dim \mathcal{P}_{2k}^n\leq (k+1) \dim \mathcal{H}_{2k}(\Z^n, S)\leq Ck^n,\quad k\geq 1.$$ Hence, the corollary indicates that the estimate in Theorem~\ref{thm:main1} is sharp for $(\Z^n, S)$ in the order of $k.$
%As a corollary, there are  $C_1,C_2$ such that
%$$C_1d^n\leq \dim H^d_n\leq C_2d^n,\quad \forall d\geq 1.$$

%The following lemma is well-known.
%\begin{lemma}\label{lem:lap}For any $k\in \Z,$ $k\geq 2,$ the map $$\Delta: {P}^k\to  {P}^{k-2}$$ is surjective.
%\end{lemma}
We prove the following lemma.
\begin{lemma}\label{lem:app2} For any $k\in \N,$ $k\geq 2,$ the map $$\Delta-\partial_t: \hat{P}^k\to  \hat{P}^{k-2}$$ is surjective.
\end{lemma}
\begin{proof} It suffices to show that for any monomial $g$ in $ \hat{P}^{k-2},$ say $$g=x_1^{a_1}x_2^{a_2}\cdots x_n^{a_n}t^b,\quad a_1,a_2,\cdots a_n,b\in \Z_+, \sum_{i=1}^n a_i+2b\leq k-2,$$ there is  $u\in  \hat{P}^k$ such that 
\begin{equation}\label{eq:eq3}(\Delta-\partial_t)u=g.\end{equation} We set $u(x,t)=\sum_{i=0}^{b+1}p_i(x)t^i,$ where $p_i$ are polynomials of degree at most $k-2i,$ to be determined later. Then the above equation is equivalent to the following:
\begin{eqnarray}\label{eq:eq2}&&\Delta p_{b+1}=0,\nonumber\\
&&\Delta p_{b}=(b+1)p_{b+1}+x_1^{a_1}x_2^{a_2}\cdots x_n^{a_n},\\
&&\Delta p_{b-1}=bp_{b},\nonumber\\
&&\cdots\nonumber\\
&&\Delta p_0= p_1.\nonumber
\end{eqnarray}We set $p_{b+1}=1.$ The following fact is useful, see \cite[Corollary~2]{HJLAGAG13}: For any $k\in \N,$ $k\geq 2,$ the map $$\Delta: {P}^k\to  {P}^{k-2}$$ is surjective. Hence there exists $p_b\in P^{k-2b}$ such that \eqref{eq:eq2} holds.
Using the above fact, we solve the above equations recursively, and obtain $p_i\in P^{k-2i}$ for all $0\leq i\leq b.$ This yields the desired polynomial $u$ solving \eqref{eq:eq3}, and proves the lemma.
\end{proof}

\begin{proof}[Proof of Theorem~\ref{thm:app1}]
By Lemma~\ref{lem:app2},
\begin{equation}\label{eq:1}\dim \mathcal{P}_k^n=\dim  \hat{P}^k-\dim  \hat{P}^{k-2}.\end{equation}

%\begin{thm} 
%\end{thm}
For any $k\in \Z_+,$ we denote by $S^k$ the dimension of the space of polynomials of degree $k$ in $\R^n,$ and by
$\hat{S}^k$ the dimension of the space of polynomials in $(x,t)\in \R^n\times \R$ of parabolic degree $k.$ Then 
$$\dim P^k=\sum_{i=0}^k S^i,\quad \dim \hat{P}^k=\sum_{i=0}^k \hat{S}^i.$$ By \eqref{eq:1},
$$\dim \mathcal{P}_k^n=\hat{S}^k+\hat{S}^{k-1}.$$ For any polynomial $u$ in $\R^n\times \R$ of parabolic degree $i,$ one can write it as
$$u(x,t)=\sum_{j=0}^{\lfloor\frac{i}{2}\rfloor}p_j(x)t^j,$$ where $p_j\in P^{i-2j}.$  Hence $$\hat{S}^i=\sum_{j=0}^{\lfloor\frac{i}{2}\rfloor} S^{i-2j}.$$
This yields that
$$\dim \mathcal{P}_k^n=\hat{S}^k+\hat{S}^{k-1}=\sum_{j=0}^{k}S^j=\dim P^k.$$ The theorem follows.
\end{proof}

%Recently, people began to study nonlinear equations on graphs, see e.g. Grigoryan/Lin/Yang. They obtained the solvability of the Kazdan-Warner and Yamabe type equations on graphs. In this paper, the authors initiate to study nonlinear Schr\"odinger equations on graphs. As is known to experts, the main difficulty for nonlinear analysis on graphs is the loss of chain rules for difference operators. So that many analytic results fail to hold on graphs. However, as observed by Grigoryan/Lin/Yang, variational methods still work. In this paper, the authors develop the Nehari methods on graphs to obtain the existence of solutions to the Schr\"odinger equations and study the limit behavior of Schr\"odinger equations as the parameter tends to the infinity. As is proved in the paper, when the parameter goes to infinity, ground state solutions of nonlinear Schr\"odinger equations converge to the ground state solution of a Yamabe type equation determined by the potential.

\bibliographystyle{alpha}
\bibliography{caloric-polynomial}

\end{document}